\documentclass[10pt]{article}

\usepackage{latexsym,amsfonts,amsmath,amsthm,amssymb, epsfig}
\usepackage[a4paper,total={170mm,250mm},vcentering,dvips]{geometry}
\usepackage{graphicx}%
\usepackage{mathrsfs}
\usepackage[normalem]{ulem}

\usepackage{dsfont}
\usepackage{enumitem}
\usepackage{csquotes}
\usepackage{bm}
\usepackage[pdftex,dvipsnames]{xcolor}
\usepackage{float}
\usepackage{caption}
\usepackage{parskip} 
\usepackage{algpseudocode}
\usepackage{algorithm}
\usepackage{subcaption}

\usepackage[hidelinks]{hyperref}
\usepackage[noabbrev,capitalize]{cleveref}

\renewcommand{\autoref}{\cref}

\newtheorem{thm}{Theorem}[section]
\newtheorem{cor}[thm]{Corollary}

\newcommand{\R}{\mathbb{R}}
\newcommand{\N}{\mathbb{N}}

\newcommand{\bit}{\begin{itemize}}
\newcommand{\eit}{\end{itemize}}

\providecommand{\U}[1]{\protect\rule{.1in}{.1in}}
\numberwithin{equation}{section}
\newtheorem{theorem}{Theorem}[section]

\newtheorem{corollary}[theorem]{Corollary}

\newtheorem{definition}[theorem]{Definition}
\newtheorem{lemma}[theorem]{Lemma}

\newtheorem{proposition}[theorem]{Proposition}
\newtheorem{remark}[theorem]{Remark}
\newtheorem{example}[theorem]{Example}

\DeclareMathOperator{\ReLU}{ReLU}

\newtheorem{defn}[theorem]{Definition}
\newtheorem{prop}[theorem]{Proposition}

\newcommand{\sgn}{\mathop{\rm sgn}\nolimits}

\newcommand{\bpr}{\begin{prop}}
\newcommand{\beq}{\begin{equation}}
\newcommand{\bd}{\begin{defn}}
\newcommand{\bex}{\begin{example}}
\newcommand{\bc}{\begin{cor}}
\newcommand{\bl}{\begin{lemma}}
\newcommand{\bt}{\begin{theorem}}
\newcommand{\br}{\begin{remark}}
\newcommand{\epr}{\end{prop}}
\newcommand{\eeq}{\end{equation}}
\newcommand{\ed}{\end{defn}}
\newcommand{\eex}{\end{example}}
\newcommand{\ec}{\end{cor}}
\newcommand{\el}{\end{lemma}}
\newcommand{\et}{\end{theorem}}
\newcommand{\er}{\end{remark}}

\newcommand{\Z}{ \mathbb{Z} }

\newcommand{\T}{ \mathbb{T} }

\newcommand{\CalC}{\mathcal C}
\newcommand{\CalS}{\mathcal S}

\begin{document}

\title{Bases of Lebesgue spaces formed by neural networks}
\author{
Vladimir Kulbatov\footnote{Department of Mathematics, Faculty of Nuclear Sciences and Physical Engineering, Czech Technical University in Prague, Trojanova 13, 12000 Praha, Czech Republic.
Email: \href{kulbavla@fjfi.cvut.cz}{kulbavla@fjfi.cvut.cz}},
Jan Lang\footnote{Department of Mathematics, The Ohio State University, 100 Math Tower, 231 West 18th Avenue, Columbus, OH 43210-1174, United States; and Department of Mathematics, Faculty of Electrical Engineering, Czech Technical University in Prague, Technick\'a~2, 166~27 Praha~6, Czech Republic.
Email: \href{lang.162@osu.edu}{lang.162@osu.edu}},
Cornelia Schneider\footnote{Friedrich-Alexander-Universit\"at Erlangen-N\"urnberg, Angewandte Mathematik III, Cauerstr. 11, 91058 Erlangen, Germany.
Email: \href{mailto:schneider@math.fau.de}{schneider@math.fau.de}}, 
and Jan Vyb\'iral\footnote{Department of Mathematics, Faculty of Nuclear Sciences and Physical Engineering, Czech Technical University in Prague, Trojanova 13, 12000 Praha, Czech Republic.
Email: \href{jan.vybiral@fjfi.cvut.cz}{jan.vybiral@fjfi.cvut.cz}.
The work of this author has been supported by the grant P202/23/04720S of the Grant Agency of the Czech Republic. J.V. is a member of the Nečas center for mathematical modeling.}
}
\maketitle

\vskip1cm

\centerline{\textit{Dedicated to Hans Triebel, on the occasion of his 90th birthday.}}

\vskip1.5cm

\begin{abstract}
The seminal work of Daubechies, DeVore, Foucart, Hanin, and Petrova introduced in 2022 a sequence of univariate piece-wise linear functions,
which resemble the classical Fourier basis and which, at the same time, can be easily reproduced by artificial neural networks
with $\ReLU$ activation function. We give an alternative way how to calculate the inner products of functions from this system
and discuss the spectral properties of the Gram matrix generated by this system.
The univariate system was later generalized to the multivariate setting by two of the authors of this work.
Instead of the usual tensor product construction, this generalization relied on the inner products inside of the argument
of the univariate sequence. It turned out that such a system forms a Riesz basis of $L_2(0,1)^n$ for every $n\ge 1$
with Riesz constants independent of $n$.

In this work, we investigate the properties of these new sequences of functions in $L_q(0,1)^n$ for $q\not =2.$
First, we show that the univariate system is a Schauder basis in $L_q(0,1)$ for every $1<q<\infty$. By a general argument, it follows
that the tensor products of this system also form  a Schauder basis in $L_q(0,1)^n$ for every $n\ge 2$ and $1<q<\infty.$
The same fact can also be  shown by measuring the distance of the tensor product system to the classical multivariate Fourier basis,
but - surprisingly - this argument only works for $n\le 3$.
If, on the other hand, we replace the outer tensor products by inner products directly in the argument of the univariate system, the same approach
is applicable for an arbitrary dimension $n\in\N.$\\

\noindent{\em Key Words: } Riesz basis, Schauder basis, Artificial Neural Networks (ANN), Tensor product\\
{\em MSC2020 Math Subject Classifications: Primary 42C15; Secondary 46B15, 68T07} 
\end{abstract}

\section{Introduction}

In their groundbreaking paper \cite{DDFHP}, Daubechies, DeVore, Foucart, Hanin, and Petrova introduced a sequence of univariate  piece-wise linear functions,
which, on one hand, can be easily reproduced by (deep) ReLU neural networks and, on the other hand, resemble the standard Fourier basis.
The topic was further developed in \cite{SV}, where a multivariate analogue of this basis was investigated.

To recall the known results, we first define two auxiliary functions
\[
\CalC(x)=4\left|x-\frac{1}{2}\right|-1\quad\text{and}\quad \CalS(x)=\left|2-4\left|x-\frac{1}{4}\right|\right|-1,\quad x\in [0,1].
\]
These functions are extended to $x\in\R$ periodically, i.e., ${\CalC}(x)=\CalC(x-\lfloor x\rfloor)$ and $\CalS(x)=\CalS(x-\lfloor x\rfloor)$.
Finally, we consider integer dilates of these functions formally defined as $\CalC_k(x):=\CalC(kx)$ and $\CalS_k(x):=\CalS(kx)$ for $k\in\N$ and $x\in\R.$

\begin{figure}[h]
\begin{minipage}{0.5\textwidth}
\includegraphics[width=10cm]{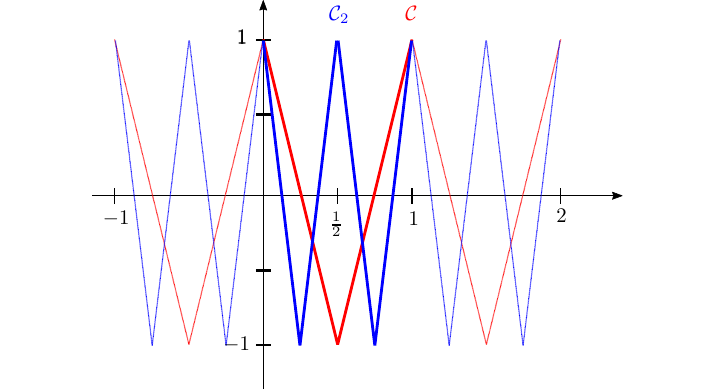}
\end{minipage}\hfill 
\begin{minipage}{0.5\textwidth}
\includegraphics[width=10cm]{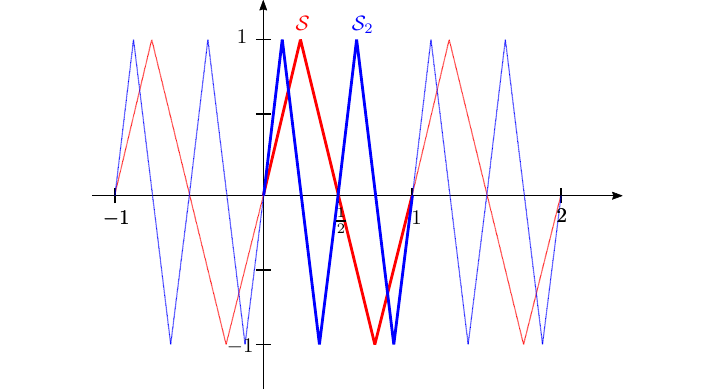}
\end{minipage}
\caption{The plot of $\CalC$, $\CalS$, $\CalC_2$ and $\CalS_2$.}
\label{fig:1}
\end{figure}

It was shown in \cite{DDFHP} (with an alternative proof given in \cite{SV}) that
\begin{equation}\label{eq:R1}
{\mathcal R}_1:=\{1\}\cup\{\sqrt{3}\,\CalC_k,\sqrt{3}\,\CalS_k: k\in \N\} 
\end{equation}
forms a so-called Riesz basis of $L_2(0,1)$. Recall, that a sequence $(\varphi_j)_{j=1}^\infty\subset L_2(0,1)$
is called a Riesz sequence, if there are two positive and finite constants $0<A\le B<\infty$ such that
\begin{equation}\label{eq:Riesz1}
A\sum_{j=1}^\infty \alpha_j^2\le \left\|\sum_{j=1}^\infty \alpha_j\varphi_j\right\|^2\le B\sum_{j=1}^\infty \alpha_j^2
\end{equation}
holds for every square-summable sequence $(\alpha_j)_{j=1}^\infty$.
In such a case, the constants $A$ and $B$ are usually called Riesz constants of $(\varphi_j)_{j=1}^\infty$.
Furthermore, if the closed linear span of $(\varphi_j)_{j=1}^\infty$ is the whole $L_2(0,1)$, then we call it a Riesz basis.
With this notation, it was shown in \cite{DDFHP} that the Riesz constants of \eqref{eq:R1} can be chosen as $A=\frac{1}{2}$ and $B=\frac{3}{2}.$
The main aim of this work is to complement our knowledge regarding the properties of \eqref{eq:R1}.

First, our proof that \eqref{eq:R1} forms a Riesz basis of $L_2(0,1)$ was based on the properties of the Gram matrices $(\langle \CalC_k,\CalC_l\rangle)_{k,l\in\N}$
and $(\langle \CalS_k,\CalS_l\rangle)_{k,l\in\N}$ and the values of the inner products $\langle \CalC_k,\CalC_l\rangle$
and $\langle \CalS_k,\CalS_l\rangle$ were calculated by expanding $\CalC_k$ and $\CalC_l$ (or $\CalS_k$ and $\CalS_l$) into their Fourier series.
Given the simplicity of $\CalC_k$ and $\CalS_l$, one might wonder if there is not a direct way to compute these inner products.
In Section \ref{sec:inner} we provide such a calculation, which avoids any use of Fourier analysis (but is based on the Chinese remainder theorem instead).

Our next contribution to \cite{SV} is the proof that the Gram matrices $(\langle \CalC_k,\CalC_l\rangle)_{k,l\in\N}$
and $(\langle \CalS_k,\CalS_l\rangle)_{k,l\in\N}$ have the same spectra, although their entries differ sometimes in sign.
We observed this effect already numerically  when working  on \cite{SV} and Section \ref{sec:spectra}
presents a short proof of this fact.

The next aim of our present work is to investigate the properties of \eqref{eq:R1} outside of the Hilbert space setting.
To be more specific, we consider Lebesgue spaces $L_q(0,1)$ for $1<q<\infty.$
In that case, there is a number of different ways how to define a basis and the subject was studied extensively in functional analysis \cite{AlbiacKalton,Fabian,LT,Sing}.
Our focus lies on the notion of a Schauder basis, which we recall for the reader's convenience.
\begin{definition}
If $X$ is a Banach space, then the sequence $(\varphi_j)_{j=1}^\infty\subset X$ is a Schauder basis of $X$ if for every element $x\in X$
there is a unique sequence of scalars $(\alpha_j)_{j=1}^\infty$ such that
\begin{equation}\label{eq:def_Schauder}
x=\sum_{j=1}^\infty \alpha_j \varphi_j.
\end{equation}
\end{definition}
Note, that the series in \eqref{eq:def_Schauder} does not need to converge unconditionally and that the ordering of the basis elements $(\varphi_j)_{j=1}^\infty$
may be crucial. This effect was studied in detail in harmonic analysis in connection with the convergence of multivariate Fourier series, cf. \cite{Feff1, Feff2}.
We return to this issue later on in Section \ref{Sec:3}, where we study Schauder bases of spaces of multivariate functions.

It is a classical result from harmonic analysis (cf. \cite[Example II.B.11]{Woj}) that the complex exponential functions
$(e^{2\pi ikx})_{k\in\Z}$ form a Schauder basis of $L_q(0,1)$ for every $1<q<\infty.$
To deal with function spaces of real-valued functions and to simplify the notation, we use a variant of this result, namely that the system $(\sqrt{2}\sin(j\pi t))_{j\in\N}$
is a Schauder basis of the real $L_q(0,1)$ space.
The piece-wise linear approximation of this system is defined by setting $S(t):=\CalS(t/2)$ for $t\in\R$, 
where $\CalS(x)=\CalS_1(x)$ is one of the functions from the Riesz basis \eqref{eq:R1} above, and $S_j(t)=S(jt)$ for $j\in \N$ . 
Clearly, the family $(S_j)_{j\in \N}$ is again easily reproducible by artificial neural networks, and Theorem \ref{Theorem 8.0.7.6} shows that it also constitutes
a Schauder basis in $L_q(0,1)$. The argument used to prove Theorem \ref{Theorem 8.0.7.6} is rather standard and commonly used in the context of Schauder bases.
Namely, we show that there is a bounded linear operator $T$, mapping $L_q(0,1)$ onto itself, which has a bounded inverse and which maps the sequence
$(\sqrt{2}\sin(j\pi t))_{j\in\N}$ onto $(S_j)_{j\in \N}$. This allows to transfer the property of being a Schauder basis from one of these systems to the other.

Section \ref{Sec:3} is devoted to spaces of multivariate functions, defined on $(0,1)^n.$ Let us first comment on the Hilbert space setting $q=2$ and the known results.
The usual and most natural way how to create an orthonormal basis of $L_2(0,1)^n$ -- once we have an orthonormal basis of $L_2(0,1)$ -- is to consider the tensor products
of this univariate basis. This idea can  easily be translated to Riesz bases but, unfortunately, this approach comes with  certain drawbacks. First,
the Riesz constants of the multivariate basis are the products of the Riesz constants of the univariate basis. Given the results of \cite{DDFHP}, this would lead to an 
exponential dependence of the Riesz constants on the underlying dimension $n$. Another obstacle is that  $\ReLU$ artificial neural networks 
cannot reproduce the multiplication function $f:(x,y)\to x\cdot y$ exactly, cf. \cite{EPGB,Telgarsky}, and, therefore, the tensor products of the elements of \eqref{eq:R1} are not easily
reproducible by $\ReLU$ artificial neural networks.

A more viable direction was pursued in \cite{SV}, where the multivariate analogue of \eqref{eq:R1} was defined using inner products as follows
\begin{equation}\label{eq:CS_Rn_intro}
 {\mathcal R}_n:=\{1\}\cup\{\sqrt{3}\, \CalC_{\mathbf{k}},\sqrt{3}\,\CalS_{\mathbf{k}}: \mathbf{k}\in\Z^n, \mathbf{k}+\!\!\!>0\}.
\end{equation}
Here, $\mathbf{k}+\!\!\!>0$ for $\mathbf{k}\in\Z^n\setminus\{0\}$ means that the first non-zero entry of $\mathbf{k}$ is non-negative
and $\CalC_{\mathbf{k}}(x):=\CalC(\mathbf{k}\cdot x)$ and $\CalS_{\mathbf{k}}(x):=\CalS(\mathbf{k}\cdot x)$.
Using this definition, it was shown in \cite{SV} that $\mathcal{R}_n$ is a Riesz basis of $L_2(0,1)^n$ with Riesz constants $A=1/2$ and $B=3/2$
independent of $n$. In this sense, taking the inner products of $\mathbf{k}$ and $x$ inside of the argument of $\CalC$ and $\CalS$ is clearly superior
to taking simply the tensor products of the elements of \eqref{eq:R1}. The potential of \eqref{eq:CS_Rn_intro} in approximation
of functions from Sobolev and Barron spaces was then investigated in \cite{SUV}.

Section \ref{Sec:3} investigates the behavior of the different sequences of multivariate functions outside of the Hilbert space setting,
i.e., as subsets of $L_q(0,1)^n$ for $1<q<\infty.$ First, we use the fact that the tensor products of functions from $(\sqrt{2}\sin(j\pi t))_{j\in \N}$
form a Schauder basis of $L_q(0,1)^n$ and apply the standard proof method to show in Theorem \ref{thm:123} that the same holds also for tensor products
of the elements of $(S_j)_{j\in\N}$.
Interestingly, the distance between the corresponding multivariate bases grows with the dimension and the proof method of Theorem \ref{thm:123}
only works for $n\le 3.$ In Section \ref{sec:tensor}, we apply the abstract theory of tensor products of Banach spaces
to show, that this restriction on $n$ is mainly the artifact of the proof method used, and that the same result holds also for $n>3.$

In the next step, we compare this result to the construction involving inner products. To be more precise, in Theorem \ref{thm:sincos_inner}
we show that
\begin{equation*}
\{1\}\cup\{\sqrt{2}\,\cos(2\pi \mathbf{k}\cdot x):\mathbf{k}\in\Z^n, \mathbf{k}+\!\!\!>0\}\cup\{\sqrt{2}\,\sin(2\pi \mathbf{k}\cdot x):\mathbf{k}\in\Z^n,\mathbf{k}+\!\!\!>0\}
\end{equation*}
is an orthonormal basis of $L_2(0,1)^n$. Next we show in Proposition \ref{prop:sincos} that it also constitutes a Schauder basis of $L_q(0,1)^n$
for all $n\in\N$ and $1<q<\infty.$ Finally, the proof method of Theorem \ref{Theorem 8.0.7.6} then allows us (cf. Theorem \ref{thm:Rn}) to conclude that $\mathcal{R}_n$
is a Schauder basis in $L_q(0,1)^n.$ Note that in this case this approach is independent of the dimension $n$. 

Our final remark concerns again the setting of Hilbert spaces. Without much additional effort, the proof method, which we use
to deal with Schauder bases, can also be applied to Riesz bases. When doing so, we not only re-prove the results
of \cite{DDFHP} and \cite{SV} for ${\mathcal R}_1$ and ${\mathcal R}_n$, but we also improve the lower Riesz constant
$A$ to $A=0.5787\dots>1/2$ in both, the univariate and the multivariate setting, cf. Theorem \ref{thm:Riesz_1} and Theorem \ref{thm:Riesz_n}, respectively.

\section{Univariate bases}

The aim of this section is twofold. First, we revisit some properties of the system \eqref{eq:R1}, which were already discussed in \cite{SV},
but we provide alternative proofs, which shed new light on this system and, possibly, might also allow to extend these properties to other systems in the future.
The second aim of this section is to investigate the properties of \eqref{eq:R1} outside of the Hilbert space setting, i.e., to study the properties
of \eqref{eq:R1} as a subset of $L_q(0,1)$ for $q\not=2.$

\subsection{Inner products of functions from ${\mathcal R}_1$}\label{sec:inner}

The inner products of functions from ${\mathcal R}_1$ were calculated in \cite[Lemma 2.1]{SV} and they played a crucial role
in the estimates of the Riesz constants of ${\mathcal R}_1$. The original proof method, inspired by \cite{DDFHP}, was based on the Fourier decomposition
of the elements of \eqref{eq:R1}. As the functions from \eqref{eq:R1} are piece-wise linear, it seems quite intuitive that there should also be a simpler
and more direct way to compute these inner products. In this section, we indeed present such a way, based on direct calculations and the Chinese remainder theorem \cite[Theorem 121]{HW}.

For the reader's convenience, we restate the result, which we are going to reprove.

\begin{lemma}[{\cite[Lemma 2.1]{SV}}]\label{Lem21}
Let $j,k\in\N.$ Then
\begin{enumerate}
    \item $\langle \CalC_j,\CalS_k\rangle=0$;
    \item $\langle \CalC_j,\CalC_k\rangle=\langle\CalS_j,\CalS_k\rangle=0$ if the prime factorizations of $j$ and $k$ contain a different power of 2;
    \item If the prime factorizations of $j$ and $k$ contain the same power of two, then
    \[
    3\,\langle \CalC_j,\CalC_k\rangle = 3\,|\langle \CalS_j,\CalS_k\rangle|=\frac{\gcd(j,k)^4}{j^2\cdot k^2}.
    \]
    The sign of $\langle \CalS_j,\CalS_k\rangle$ is negative if, and only if, $(j+k)/(2\gcd(j,k))$ is even.
\end{enumerate}
\end{lemma}

\begin{proof}
First, we show that it is enough to prove the lemma for $j$ and $k$ co-prime, i.e., if $\gcd(j,k)=1.$ Indeed, assume that the lemma
was already proven for co-prime indices and assume that $j$ and $k$ are arbitrary positive integers. Then we write
$j=\ell\cdot J$ and $k=\ell\cdot K$, where $\ell=\gcd(j,k)$ and $J$ and $K$ are co-prime. Using periodicity, we get
\begin{align*}
 \langle \CalC_j,\CalC_k\rangle&=\int_0^1 \CalC(jx)\CalC(kx)dx=\int_0^1 \CalC(\ell J x)\CalC(\ell Kx)dx\\
 &=\int_0^\ell {\CalC}(Jy)\CalC(Ky)\frac{dy}{\ell}=\int_0^1 {\CalC}(Jy)\CalC(Ky) dy=\langle \CalC_J,\CalC_K\rangle.
\end{align*}
If $j$ and $k$ contain a different power of two in their prime factorization, then so do $J$ and $K$ and the last expression is equal to zero.
On the other hand, if $j$ and $k$ contain the same power of two, then $J$ and $K$ are odd and assuming the lemma holds for $J$ and $K$ co-prime, we obtain
\begin{align*}
 \langle \CalC_j,\CalC_k\rangle=\langle \CalC_J,\CalC_K\rangle=\frac{1}{J^2\cdot K^2}=\frac{\ell^4}{(\ell \cdot J)^2\cdot (\ell \cdot K)^2}=
 \frac{\gcd(j,k)^4}{j^2\cdot k^2}.
\end{align*}
The same argument applies also for the calculation of $\langle \CalS_j,\CalS_k\rangle$.

\medskip

Hence, for the rest of the proof we consider positive co-prime integers $j$ and $k$ and calculate
\begin{align*}
\langle \CalC_j,\CalC_k\rangle&=\int_0^1 \CalC(jx)\CalC(kx)dx=\sum_{m=0}^{jk-1}\int_{\frac{m}{jk}}^{\frac{m+1}{jk}}\CalC(jx)\CalC(kx)dx
=\frac{1}{jk}\,\sum_{m=0}^{jk-1}\int_{m}^{m+1}\CalC(y/k)\,\CalC(y/j)dy\\
&=\frac{1}{jk}\,\sum_{m=0}^{jk-1}\int_{0}^{1}\CalC\left(\frac{m+z}{k}\right)\,\CalC\left(\frac{m+z}{j}\right)dz.
\end{align*}
For every integer $m$ with $0\le m<jk$, there is a unique pair of integers $(a,b)$ with
\begin{gather}
\notag 0\le a<j,\quad m\bmod j=a,\\
\notag 0\le b<k,\quad m\bmod k=b.
\end{gather}
The Chinese remainder theorem states that, if $j$ and $k$ are co-prime, then every such pair $(a,b)$ appears exactly once
when $m$ is running from 0 to $jk-1$. This and the periodicity of $\CalC$ allow us to write
\begin{align*}
\langle \CalC_j,\CalC_k\rangle&=\frac{1}{jk}\,\sum_{a=0}^{j-1}\sum_{b=0}^{k-1}\int_{0}^{1}\CalC\left(\frac{b+z}{k}\right)\,\CalC\left(\frac{a+z}{j}\right)dz\\
&=\frac{1}{jk}\,\int_{0}^{1} \left[\sum_{a=0}^{j-1}\CalC\left(\frac{a+z}{j}\right)\right]\,\cdot\, \left[\sum_{b=0}^{k-1}\CalC\left(\frac{b+z}{k}\right)\right]dz.
\end{align*}
Due to symmetry, it is enough to evaluate the first sum. If $j$ is even, this sum is equal to zero. Indeed, then
\begin{align*}
\sum_{a=0}^{j-1}\CalC\left(\frac{a+z}{j}\right)&=\sum_{a=0}^{j/2-1}\CalC\left(\frac{a+z}{j}\right)+\sum_{a=j/2}^{j-1}\CalC\left(\frac{a+z}{j}\right)\\
&=\sum_{a=0}^{j/2-1}\left[4\left(\frac{1}{2}-\frac{a+z}{j}\right)-1\right]+\sum_{a=j/2}^{j-1}\left[4\left(\frac{a+z}{j}-\frac{1}{2}\right)-1\right]\\
&=\sum_{a=0}^{j/2-1}\left[1-4\,\frac{a+z}{j}\right]+\sum_{a=0}^{j/2-1}\left[4\,\frac{a+j/2+z}{j}-3\right]=0.
\end{align*}
If $j$ is odd, we consider first $0\le z<1/2$ and obtain in a similar way
\begin{align}
\label{eq:sumsplit}
\sum_{a=0}^{j-1}\CalC\left(\frac{a+z}{j}\right)&=\sum_{a=0}^{\frac{j-1}{2}}\CalC\left(\frac{a+z}{j}\right)+\sum_{a=\frac{j+1}{2}}^{j-1}\CalC\left(\frac{a+z}{j}\right)\\
\notag&=\sum_{a=0}^{\frac{j-1}{2}}\left[4\left(\frac{1}{2}-\frac{a+z}{j}\right)-1\right]+\sum_{a=\frac{j+1}{2}}^{j-1}\left[4\left(\frac{a+z}{j}-\frac{1}{2}\right)-1\right]\\
\notag&=1-4\,\frac{z}{j}+\sum_{a=1}^{\frac{j-1}{2}}\left[1-4\,\frac{a+z}{j}+4\,\frac{a+\frac{j-1}{2}+z}{j}-3\right]\\
\notag &=1-4\,\frac{z}{j}+\frac{j-1}{2}\left(2\cdot \frac{j-1}{j}-2\right)=\frac{1-4z}{j}.
\end{align}
If $1/2\le z<1$, we proceed similarly and split the sum in \eqref{eq:sumsplit} into $0\le a\le \frac{j-3}{2}$ and $\frac{j-1}{2}\le a \le j-1$, which leads to
\[
\sum_{a=0}^{j-1}\CalC\left(\frac{a+z}{j}\right)=\frac{4z-3}{j}.
\]
Finally, we obtain
\[
\langle \CalC_j,\CalC_k\rangle=\frac{1}{jk}\int_0^{1/2}\frac{1-4z}{j}\cdot \frac{1-4z}{k}dz+\frac{1}{jk}\int_{1/2}^1 \frac{4z-3}{j}\cdot\frac{4z-3}{k} dz=\frac{1}{3j^2k^2}.
\]

\medskip

The mixed inner products $\langle \CalC_j,\CalS_k\rangle$ vanish due to the symmetry around the middle point $1/2.$
Indeed, $\CalC(x)=\CalC(1-x)$ and $\CalS(x)=-\CalS(1-x)$ for $x\in (0,1)$ and this relation extends easily also to $\CalC_j$ and $\CalS_k.$
Consequently, $\langle \CalC_j,\CalS_k\rangle=0$ for all $j,k\in\N.$

\medskip

Finally, the inner products $\langle \CalS_j,\CalS_k\rangle$ could be calculated similarly to the method used above for $\langle \CalC_j,\CalC_k\rangle$,
but they can actually be derived easily from what we have shown already. For example, if $j$ and $k$ are both odd, we get $\CalS(jx)=\CalC(jx-1/4)=\CalC(jx-j/4)$
if $4$ divides $j-1$ and $\CalS(jx)=-\CalC(jx-3/4)=\CalC(jx-j/4)$ if $4$ divides $j-3.$ Consequently,
\begin{align*}
\langle \CalS_j,\CalS_k\rangle&=\int_0^1 \CalS(jx)\CalS(kx)dx = \int_{0}^{1} \CalC(j(x-1/4))\CalC(k(x-1/4))dx=\int_0^1 \CalC(jx)\CalC(kx)dx=\langle \CalC_j,\CalC_k\rangle
\end{align*}
if $4$ divides both $j-1$ and $k-1$. Similarly, one can argue if 4 divides $j-3$ and $k-3$. In both these cases $\frac{j+k}{2\gcd(j,k)}$ is odd and the result follows.
The other cases are treated similarly.
\end{proof}

\subsection{Spectral properties of ${\mathcal R}_1$}\label{sec:spectra}

The Gram matrix of \eqref{eq:R1} played a crucial role in \cite{SV} as its spectrum determines the Riesz constants of \eqref{eq:R1}.
Also, it was observed that it is somehow simpler to study first the spectrum of the Gram matrix of the truncated system
\[
{\mathcal R}_1^N=\{1\}\cup\{\sqrt{3}\,\CalC_k,\sqrt{3}\,\CalS_k:1\le k\le N\}
\]
and then pass to the limit $N\to\infty.$ It is quite natural to re-order the Gram matrix of ${\mathcal R}_1^N$ as
\[
G_N=\left[\begin{matrix}1&0&0\\0&G_N^{\CalC}&0\\0&0&G_N^{\CalS}\end{matrix}\right],
\]
where $G_N^{\CalC}=(\langle \CalC_j,\CalC_k\rangle)_{j,k=1}^N$ is the Gram matrix of $\{\CalC_k:1\le k\le N\}$ and
$G_N^{\CalS}=(\langle \CalS_j,\CalS_k\rangle)_{j,k=1}^N$ is the Gram matrix of $\{\CalS_k:1\le k\le N\}$.
By Lemma \ref{Lem21}, the entries of $G_N^{\CalC}$ and $G_N^{\CalS}$ are equal in the absolute value, but they might have a different sign.
Interestingly, when working on \cite{SV}, numerical evidence suggested that the spectra of $G_N^{\CalC}$ and $G_N^{\CalS}$
are equal. The aim of this section is to explain this phenomenon.

\begin{theorem}
Let $N \ge 1$. Then the spectra of $G_N^{\CalC}$ and $G_N^{\CalS}$ coincide.
\end{theorem}

\begin{proof}
We use the fact (cf. Lemma \ref{Lem21}) that $\langle \CalC_j,\CalC_k\rangle=\langle \CalS_j,\CalS_k\rangle=0$
if the prime factorizations of $j$ and $k$ contain a different power of two. This allows us to re-group the indices $j\in\{1,\dots,N\}$
in such a way that $G_N^{\CalC}$ becomes again a block-diagonal matrix. We put for $0\le r\le \lfloor \log_2 N\rfloor$
\[
I_N^r=\{1\le j\le N: 2^r\mid j\ \text{and}\ 2^{r+1}\nmid j\}
\]
and denote $G_{N,r}^{\CalC}=(\langle \CalC_j,\CalC_k\rangle)_{j,k\in I_N^r}$ (and similarly for $G_{N,r}^{\CalS}$).
Then we can write
\[
G_N^{\CalC}=\left[\begin{matrix}G_{N,0}^{\CalC}&0&0&\dots\\0&G_{N,1}^{\CalC}&0&\dots\\0&0&G_{N,2}^{\CalC}&\dots\\0&0&0&\ddots\end{matrix}\right]
\]
and, again, similarly for $G_N^{\CalS}$. We show that $G_{N,r}^{\CalC}$ and $G_{N,r}^{\CalS}$ have the same spectrum for all $r$'s.

We start with $r=0$. Note, that $I_N^0$ collects all odd integers between 1 and $N$. We show that $G_{N,0}^{\CalS}=D_{N,0} G_{N,0}^{\CalC} D_{N,0}$,
where $D_{N,0}$ is a diagonal matrix with entries $\varepsilon_1,\varepsilon_3,\dots\in\{-1,+1\}$ on the diagonal, where we put
\[
\varepsilon_j:=\begin{cases}+1,&\quad\text{if}\ 4\mid(j-1),\\ -1,&\quad \text{if}\ 4\mid(j-3).\end{cases}
\]
Indeed, with this choice we have
\[
\varepsilon_j\cdot\varepsilon_k=\begin{cases}+1,&\quad\text{if}\ 4\mid (j+k-2),\\
-1,&\quad\text{if}\ 4\mid(j+k).
\end{cases}
\]
As $\gcd(j,k)$ is odd, $\varepsilon_j\cdot\varepsilon_k=-1$ if, and only if, $\frac{j+k}{2\gcd(j,k)}$ is even.
Consequently, by Lemma \ref{Lem21},
\[
\langle \CalS_j,\CalS_k\rangle=\varepsilon_j\cdot\varepsilon_k\cdot\langle \CalC_j,\CalC_k\rangle.
\]
Hence, $G_{N,0}^{\CalS}=D_{N,0} G_{N,0}^{\CalC} D_{N,0}$ and the result follows.

If $r\ge 1$, a similar argument goes through by setting $\varepsilon_j=\varepsilon_{j/2^r}$ for $j\in I_N^r.$ We leave the details to the reader.
\end{proof}

\subsection{Schauder bases of Lebesgue spaces}\label{sec:bases_leb}

The aim of this section is to extend our study of bases of piece-wise linear functions to Lebesgue spaces $L_q(0,1)$ with $1<q<\infty.$
To simplify the notation, we use the following standard trick. If $f$ is an (integrable) function defined on $(0,1)$, we consider its odd extension (denoted again by $f$) to $(-1,1)$
and consider its Fourier series expansion into the basis formed by $\{1\}\cup\{\cos(j\pi x),\sin(j\pi x):j\in\N\}$.
Due to the symmetry of this extension, the $\cos$-part of this expansion vanishes, which allows us to represent $f$ on $(0,1)$ as its \emph{Fourier sine series},
i.e., in terms of $\sin(j\pi x)$ with $j\in\N.$

It is a well known result (essentially due to Riesz \cite{Riesz} who investigated the convergence of partial Fourier sums in Lebesgue spaces, see also \cite{EL11,Grafakos,Zyg})
that  $(\sin(j\pi\cdot))_{j\in\mathbb{N}}$ is not only an orthogonal basis of $L_2(0,1)$ but also a Schauder basis of $L_{q}(0,1)$ for any $q\in(1,\infty)$. In the sequel we abbreviate
\begin{equation}\label{Eq. 2.2.2}
e_{j}(t):=\sqrt{2}\sin(j\pi t),\quad j\in\N,\quad t\in\R. 
\end{equation}

\begin{minipage}{0.45\textwidth}
We now introduce the piece-wise linear analogue of \eqref{Eq. 2.2.2}.
First, we observe that the hat function
$$S(t)= 
\begin{cases} 
2t & \text{for } 0 \leq t \leq \frac{1}{2}, \\
2(1 - t) & \text{for } \frac{1}{2} < t \leq 1,  \\
\end{cases}
$$
is a piece-wise linear approximation of $\sin(\pi t)$ and that $S(t)$ and $\sin(\pi t)$ coincide for $t\in\{0,1/2,1\}$.
Then, $S$ is extended by the anti-symmetric reflection rule to $[0,2]$ and, finally, periodically to $\R$ with period 2, i.e., 
\[
S(t)=(-1)^{\lfloor t\rfloor}\,S(\{t\}),\quad t\in \R,
\]
where $t=\lfloor t\rfloor + \{t\}$ is the decomposition of $t$ into its integer and fractional part.

\end{minipage}\hfill \begin{minipage}{0.45\textwidth}
\includegraphics[width=6cm]{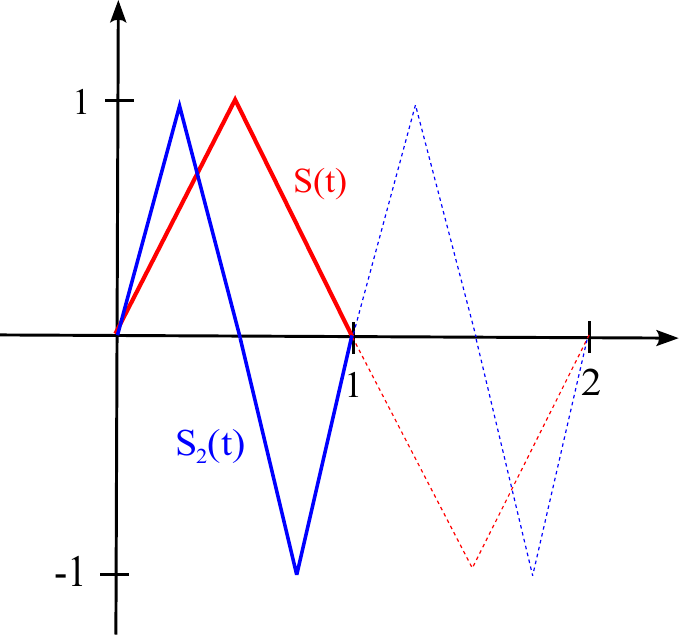}
\captionof{figure}{Hat function $S(t)$ and scaled version $S_2(t)$}
\end{minipage}\\

With this extension we define 
\begin{equation}
S_{j}(t)=S(jt),  \quad j\in\mathbb{N}, \quad t\in \R. \label{Equation 8.0.1.1}%
\end{equation}

Intuitively, the system $(S_j)_{j\in\N}$ is a piece-wise linear analogue of $(e_j)_{j\in\N}$.
Indeed, $S_j(t)$ interpolates $e_j(t)$ at points $t\in\{\frac{0}{2^j},\frac{1}{2^j},\dots,\frac{2^j}{2^j}\}.$
The main aim of this section is to prove Theorem \ref{Theorem 8.0.7.6}, i.e., to show that the system $(S_j)_{j\in \N}$ is a Schauder basis in $L_q(0,1)$ for every $1<q<\infty$.
The proof is based on the following observation (and we refer to \cite[p. 75]{Hig} for details).

Let us assume that $(\varphi_j)_{j\in\N}$ and $(\psi_j)_{j\in\N}$ are two sequences in a Banach space $X$. Further, we assume that there is a bounded and invertible linear operator
$T$ mapping $X$ onto $X$ such that $T(\varphi_j)=\psi_j$ for every $j\in\N.$
Then it holds that if $(\varphi_j)_{j\in\N}$ is a Schauder basis of $X$ then the same is true also for $(\psi_j)_{j\in\N}$.

The main idea behind this observation is rather straightforward: for every $f\in X$, the expansion of $T^{-1}f$
\[
T^{-1}f=\sum_{j=1}^{\infty}\alpha_j\varphi_j
\]
is transformed by $T$ into
\[
f=T\left(\sum_{j=1}^{\infty}\alpha_j\varphi_j\right)=\sum_{j=1}^{\infty}\alpha_j\, T\varphi_j=\sum_{j=1}^{\infty}\alpha_j\, \psi_j.
\]

\begin{remark}
Note that compared to the Riesz basis $\mathcal{R}_1$ from \eqref{eq:R1} we  now study one scale of functions $\{S_j\}_{j\in \mathbb{N}}$ only. In particular, the scales partly coincide, i.e., we have  $S_2 =\CalS $ which implies  $S_{2m}(\cdot)=S_2(m\cdot)=\mathcal{S}(m\cdot)=\mathcal{S}_m(\cdot)$. 
\end{remark}

The link between the two sequences  is given by the expansion of $S(t)=S_1(t)$ which is 
\begin{equation}
S(t)=\sum_{k=1}^{\infty}\widehat{S}(k)e_k(t)=\sqrt{2}\sum_{k=1}^{\infty}\widehat{S}(k)\sin(k\pi t), \label{Equation 8.0.1.3}
\end{equation}
where
\begin{equation}
\tau_k:=\widehat{S}(k)=\int_0^1 S(t)e_k(t)dt=\sqrt{2}\int_{0}^{1}S(t)\sin(k\pi t)dt. \label{Equation 8.0.1.3'}
\end{equation}

The behavior of the Fourier coefficients $(\tau_k)_{k\in \N}$ is crucial for our later studies. Due to the simple form of the function $S(t)$, they can be easily calculated.

\begin{proposition}\label{prop:1.5} For the sequence $(\tau_k)_{k\in \mathbb{N}}$ it holds   
\[
\tau_{2k-1}=(-1)^{k-1}\frac{4\sqrt{2}}{\pi^2} \frac{1}{(2k-1)^2} \quad \text{and}\quad  \tau_{2k}=0 \quad \text{for}\quad  k \in \mathbb{N}.
\]
\end{proposition}

\begin{proof} Using integration by parts we compute 
\begin{eqnarray*}
\tau_{2k-1} =  \int_0^1 S(t) e_{2k-1}(t) \, dt
= 2\sqrt{2} \int_0^{1/2} 2t \sin((2k-1)\pi t) \, dt
= (-1)^{k-1}\frac{4\sqrt{2}}{(2k-1)^2\pi^2},
\end{eqnarray*}
which gives the desired result. Moreover, we obtain $\tau_{2k}=0$ by the symmetry of $S$ around the point $t=1/2$.
\end{proof}

Let us now consider a general function $f$ on $[0,1]$. First, we extend it to $(-1,1]$ by the anti-symmetric reflection rule, i.e., we  put 
\[
f(t):=-f(-t),\quad -1<t<0.
\]
Finally, we extend it to $\R$ (still denoting it by $f$) periodically with period 2.
Alternatively, if $t\in \R$ and $t=\lfloor t \rfloor+\{t\}$ is the decomposition of $t$ into its integer and fractional part, the extension of $f$ is given by
\[
f(t):=(-1)^{\lfloor t\rfloor} f(1-|1-2\{t/2\}|), \quad t\in\R. 
\]

We now construct a linear homeomorphism $T$ of $L_{q}(0,1)$ onto itself that maps each $e_{j}$ to $S_{j}$.
The definition is based on \eqref{Equation 8.0.1.3} and \eqref{Equation 8.0.1.3'}. Essentially, we need that $T$
maps $e_1$ to $S_1$ and that $T$ commutes with integer dilations. 
Once this is done it will follow from the general considerations sketched above that
the family $(S_{j})_{j\in\N}$ forms a basis of $L_{q}(0,1).$

\begin{lemma}
\label{Lemma 2.2.6} Let $q\in(1,\infty)$ and define the map $T$   by%
\begin{equation}
Tg(t)=\sum_{m=1}^{\infty}\tau_{m}g(mt), \qquad t\in (0,1). \label{Equation 8.0.1.10}%
\end{equation} 
Then $T$ has the following properties:
\bit 
\item[(i)] The map $T:L_{q}(0,1)\rightarrow$ $L_{q}(0,1)$ is bounded and linear  with $\left\Vert T\right\Vert \leq \frac{1}{\sqrt{2}}.$ 
\item[(ii)]  $Te_{j}=S_{j}$  for all $j\in\mathbb{N}$. 
\item[(iii)] $T$ has a bounded inverse.
\eit 
\end{lemma}

\begin{proof}
\bit 
\item[(i)] Linearity is clear. 
For every $g\in L_{q}(0,1)$ it holds 
\begin{align}\label{eq-isometry}
\|g(m\cdot)\|_{q}^q=\int_{0}^{1}\left\vert g(mt)\right\vert ^{q}dt  &  =m^{-1}\int_{0}%
^{m}\left\vert {g}(s)\right\vert ^{q}ds=m^{-1}\sum_{k=1}^{m}%
\int_{k-1}^{k}\left\vert {g}(s)\right\vert ^{q}ds
=\int_{0}^{1}\left\vert g(s)\right\vert ^{q}ds=\|g\|_q^q, 
\end{align}
and by the triangle inequality and Proposition \ref{prop:1.5}, we  obtain  
\[
\left\Vert T\right\Vert \leq  \sum_{k=1}^{\infty}|\tau_{2k-1}|= \sum_{k=1}^{\infty}\frac{4\sqrt{2}}{(2k-1)^{2}%
\pi^{2}}=\frac{1}{\sqrt{2}}.
\]
\item[(ii)] Next we show that $Te_j=S_j$ for every $j\in\N$. For that purpose, we calculate
\begin{align}
\notag S_j&=\sum_{k=1}^\infty \widehat{S_j}(k)e_k=\sqrt{2}\sum_{k=1}^\infty e_k\int_0^1 S_j(t)\sin(k\pi t)dt=\sqrt{2}\sum_{k=1}^\infty e_k\int_0^1 S(jt)\sin(k\pi t)dt\\
\notag&=2\sum_{k=1}^\infty e_k\int_0^1 \sum_{m=1}^\infty \widehat{S}(m)\sin(jm\pi t)\sin(k\pi t)dt
=2\sum_{k=1}^\infty e_k\sum_{m=1}^\infty\widehat{S}(m)\int_0^1  \sin(jm\pi t)\sin(k\pi t)dt\\
\label{eq:T1} &=2\sum_{m=1}^\infty\widehat{S}(m)\sum_{k=1}^\infty e_k\int_0^1  \sin(jm\pi t)\sin(k\pi t)dt=\sum_{m=1}^\infty\widehat{S}(m) e_{jm}=\sum_{m=1}^\infty\tau_m e_{jm}
=\sum_{m=1}^\infty\tau_m e_j({m\,\cdot})=Te_j.
\end{align}

\item[(iii)] By the definition of $T$ and \eqref{eq-isometry} we observe that
\[
\left\Vert T-\tau_{1}\text{Id}\right\Vert \leq\sum_{j=1}^{\infty}\left\vert
\tau_{2j+1}\right\vert ,
\]
and so the invertibility of $T$ will follow from Theorem II.1.2  of \cite{Yos} (Neumann Series)
if we can show that%
\begin{equation}
\sum_{j=1}^{\infty}\left\vert \tau_{2j+1}\right\vert <\left\vert \tau
_{1}\right\vert.  \label{Equation 8.0.1.11}%
\end{equation}

From Proposition \ref{prop:1.5} we have 
\begin{equation}
\sum_{j=1}^{\infty}\left\vert \tau_{2j+1}\right\vert \leq\frac{4\sqrt{2}%
}{\pi^{2}}\left(  \frac{\pi^{2}}{8}-1\right)< \frac{4\sqrt{2}}{\pi^2}=\tau_1, \label{Equation 8.0.1.12}%
\end{equation}
from which (\ref{Equation 8.0.1.11}) follows.
\eit 
\end{proof}

We conclude with the main result of this section.  

\begin{theorem}\label{Theorem 8.0.7.6}
Let  $q\in(1,\infty).$ Then the family $(S_{j})_{j\in\mathbb{N}}$ forms a Schauder basis of $L_{q}(0,1)$. 
\end{theorem}

\begin{proof}
Since the sequence $(e_{j})_{j\in\N}$ forms a basis of $L_{q}(0,1)$ and $T$ is a linear
homeomorphism of $L_{q}(0,1)$ onto itself with $Te_{j}=S_j$, $j\in
\mathbb{N}$,  it follows from \cite[p. 75]{Hig}, or \cite[Theorem 3.1, p. 20]{Sing},
that the sequence $(S_{j})_{j\in\N}$ forms also a Schauder basis of $L_{q}(0,1).$
\end{proof}

With only a minor modification of the argument, we can apply the same procedure also to the Hilbert space setting and
Riesz bases. Recall, that it was shown already in \cite{DDFHP} (with an alternative proof provided in \cite{SV}) that ${\mathcal R}_1$
is a Riesz basis with Riesz constants $A=1/2$ and $B=3/2$.
It is quite interesting that a straightforward modification of the proof of Theorem \ref{Theorem 8.0.7.6} (see \cite[Section VI.2.]{GK} for details)
provides the same conclusion with an improved lower Riesz constant $A$.

\begin{theorem}\label{thm:Riesz_1}
The sequence $(\sqrt{3}\, S_j)_{j\in\N}$ is a Riesz basis of $L_2(0,1)$
with Riesz constants  
\[
A:=\left\{\frac{4\sqrt{6}}{\pi^2}\left(2-\frac{\pi^2}{8}\right)\right\}^2=0.5787\dots\quad\text{and}\quad B:=\frac 32.
\]
\end{theorem}

\begin{proof}
The proof is based on the following observation. Let $(\varphi_j)_{j\in\N}$ be an orthonormal basis of a (separable) Hilbert space $H$ and let $(\psi_j)_{j\in\N}$ be a sequence of elements of $H$.
Furthermore, let us assume that there is a bounded and invertible linear mapping $R$ of $H$ onto itself, such that $R\varphi_j=\psi_j$ for every $j\in\N.$
Then
\begin{align*}
\left\|\sum_{j=1}^\infty\alpha_j\psi_j\right\|^2=\left\|\sum_{j=1}^\infty\alpha_jR(\varphi_j)\right\|^2=\left\|R\left(\sum_{j=1}^\infty\alpha_j\varphi_j\right)\right\|^2\le
\|R\|^2\cdot \left\|\sum_{j=1}^\infty\alpha_j\varphi_j\right\|^2=\|R\|^2\cdot \sum_{j=1}^\infty \alpha_j^2
\end{align*}
and, similarly,
\begin{align*}
\left\|\sum_{j=1}^\infty\alpha_j\psi_j\right\|^2
 =\left\|R\left(\sum_{j=1}^\infty\alpha_j\varphi_j\right)\right\|^2
\ge 1/\|R^{-1}\|^2\cdot \sum_{j=1}^\infty \alpha_j^2.
\end{align*}
Hence, $(\psi_j)_{j\in\N}$ is a Riesz basis and its Riesz constants can be chosen as $A=1/\|R^{-1}\|^2$
and $B=\|R\|^2.$

\medskip

To apply this approach to our setting, we define $R:=\sqrt{3}\,T$ where $T$ was defined in \eqref{Equation 8.0.1.10}. With this notation, $R(e_j)=\sqrt{3}\,S_j$
and $(\sqrt{3}\, S_j)_{j\in\N}$ is correctly normalized. Then we obtain, that $(\sqrt{3}\, S_j)_{j\in\N}$ is a Riesz basis of $L_2(0,1)$
with Riesz constants $B:=\|R\|^2=3\,\|T\|^2\le 3/2$ and $A:=1/\|R^{-1}\|^2.$ To bound $A$, we rewrite $R$ as $R=\sqrt{3}\tau_1 \mathrm{Id}-M$,
where
\[
Mg(t)=-\sum_{m=1}^\infty \sqrt{3}\tau_{2m+1}g(m\cdot t).
\]
Therefore,
\[
\|M\|\le \sqrt{3}\sum_{m=1}^\infty|\tau_{2m+1}|=\frac{4\sqrt{6}}{\pi^2}\left(\frac{\pi^2}{8}-1\right).
\]
Finally, we denote $a:=\sqrt{3}\tau_1$ and the Neumann series technique (applicable since $\|M\|/a\leq \frac{\pi^2}{8}-1<1$) gives
\begin{align*}
a\|R^{-1}\|=\left\|(R/a)^{-1}\right\|=\left\|(\mathrm{Id}-M/a)^{-1}\right\|\le \sum_{m=0}^\infty\frac{\|M\|^m}{a^m}=\frac{1}{1-\|M\|/a}.
\end{align*}
This gives
\[
A=1/\|R^{-1}\|^2\ge \big(a-\|M\|\big)^2\ge\left\{\frac{4\sqrt{6}}{\pi^2}\left(2-\frac{\pi^2}{8}\right)\right\}^2=0.5787\dots.
\]
\end{proof}

\section{Multivariate bases}\label{Sec:3}

The main aim of this section is to investigate the multivariate counterparts of the results presented in Section \ref{sec:bases_leb}.
We start with the tensor product analogue of the basis $(S_j)_{j\in\N}$.

\subsection{Tensor product basis in $L_q(0,1)^n$ generated by \( S \)}

Throughout this section we shall again assume that \(1 < q < \infty\). 
Our aim is to prove (cf. Theorem \ref{thm:123}) that the tensor products of functions from the family $(S_{j})_{j\in\mathbb{N}}$   form a basis in the Lebesgue space
$L_q(0, 1)^n$, provided $n \leq  3$.

First, we introduce the notation connected to multivariate functions and indices, cf. \cite{Lang17}. Let \[ \mathbf{m} = (m_1, \ldots, m_n), \quad \mathbf{k} = (k_1, \ldots, k_n) \in \mathbb{N}^n \] be multi-indices, and let \( x = (x_1, \ldots, x_n) \in \mathbb{R}^n \).
We write
\[
\mathbf{m}x = (m_1 x_1, \ldots, m_n x_n), \quad \mathbf{m}\mathbf{k} = (m_1 k_1, \ldots, m_n k_n)
\]
and denote
\[
\mathbf{m} \leq \mathbf{k} \quad \text{if} \quad m_i \leq k_i \quad \text{for each} \quad i \in \{1, \ldots, n\}. 
\]
Moreover,  we set
\[
\mathbf{1} := (1, \ldots, 1).
\]
For a given function \( \varphi : \R \to \mathbb{R} \), \( \mathbf{m} \in \mathbb{N}^n \), and \( x \in \mathbb{R}^n \), we put 
\begin{equation}\label{eq:0.1}
\varphi_{\mathbf{m}}(x) := \varphi(m_1 x_1) \cdots \varphi(m_n x_n).
\end{equation}
Below we choose $\varphi(t)=S(t)$, $t\in \mathbb{R}$, and study the properties of the system $(S_{\mathbf{m}})_{\mathbf{m}\in\N^n}$. Note that later on we replace the tensor products by inner products and consider also systems
\begin{equation}\label{eq:1.3a}
\varphi(\mathbf{m}\cdot x)=\varphi(m_1x_1+\ldots + m_nx_n), \quad x \in \mathbb{R}^n,
\end{equation}
i.e., $\mathbf{m}\cdot x$ is the scalar product, since these functions can be reproduced by NNs in an efficient way where we do not suffer the 'curse of dimension'.

\begin{remark}\label{def:Schauder_n} Recall that the system \( \{\varphi_{\mathbf{m}}\}_{\mathbf{m} \in \mathbb{N}^n} \) is a Schauder basis in \( L_q(0,1)^n \), \( q \in (1,\infty) \), if,
and only if, for any \( f \in L_q(0,1)^n \), there exists a unique sequence \( \{\alpha_{\mathbf{m}}\}_{\mathbf{m} \in \mathbb{N}^n} \) of scalars such that
\begin{equation}\label{eq:0.2}
f = \sum_{\mathbf{m} \in \mathbb{N}^n} \alpha_{\mathbf{m}} \varphi_{\mathbf{m}} \quad \text{in} \quad L_q(0,1)^n. 
\end{equation}
The convergence in \eqref{eq:0.2} is interpreted as 
\[
\lim_{\min\{m_1, \ldots, m_n\} \to \infty} \left\| f - \sum_{\mathbf{k} \leq \mathbf{m}} \alpha_{\mathbf{k}} \varphi_{\mathbf{k}} \right\|_{L_q(0,1)^n} = 0,
\]
which is sometimes referred to as the convergence in the \emph{Pringsheim sense}.

\end{remark}

Let us recall, cf. \eqref{Eq. 2.2.2}, that for \( e(t) = \sqrt{2} \sin(\pi t) \), \( t \in \mathbb{R} \), we denoted
\[
e_j(t) = e(jt) = \sqrt{2} \sin(j \pi t), \quad t \in \mathbb{R}, \quad j\in \mathbb{N}.
\]
We use that the system \( (e_j(t))_{j\in\N} \) is an orthonormal basis in \( L_2(0,1) \) and consider the corresponding tensor product orthonormal basis
$(e_{\mathbf{m}})_{\mathbf{m}\in\N^n}$ in \( L_2(0,1)^n \), which is defined as
\begin{equation}\label{eq:em}
e_{\mathbf{m}}(x) = e_{m_1}(x_1) \cdots e_{m_n}(x_n) = 2^{n/2} \sin(\pi m_1 x_1) \cdots \sin(\pi m_n x_n), \quad x \in \mathbb{R}^n, \quad \mathbf{m} \in \mathbb{N}^n.
\end{equation}

It is very well known that the sequence \( (e_{\mathbf{m}})_{\mathbf{m} \in \mathbb{N}^n} \)
is orthonormal in \( L_2(0,1)^n \). Indeed, the orthonormality is straightforward and
the completeness of this system in $L_2(0,1)^n$ follows from \cite[Sect. II.4, Prop. 2 and p.~50]{ReSi}. In what follows, using classical results from the analysis of multivariate Fourier series, we show that
 \( (e_{\mathbf{m}})_{\mathbf{m} \in \mathbb{N}^n} \) is a Schauder basis in the Lebesgue space \( L_q(0,1)^n \).

\begin{proposition}\label{prop:1.3} Let  \( f \in L_q(0,1)^n \), where $n\ge 1$ and  \( q \in (1,\infty) \), and set
\begin{equation}\label{eq:1.1}
\hat{f}(\mathbf{k}) := \int_{(0,1)^n} f(x) e_{\mathbf{k}}(x) \, dx, \quad \mathbf{k} = (k_1, \ldots, k_n) \in \mathbb{N}^n.
\end{equation}
Then
\begin{equation}\label{eq:1.2}
f = \sum_{\mathbf{m} \in \mathbb{N}^n} \hat{f}(\mathbf{m}) e_{\mathbf{m}}
\end{equation}
in the sense of \eqref{eq:0.2}.
\end{proposition}
\begin{proof}
Extend \( f \in L_q(0,1)^n \) in each of its variables to \((-1,1)\) by {\em oddness} (i.e., reflect $f$ across each coordinate axis with a sign flip). Thus,   define $f:(-1,1)^n\rightarrow \mathbb{R}$ by 
\begin{equation}\label{extend-odd}
f(x_1,\ldots, x_n):=\sgn(x_1)\dots \sgn(x_n)f(|x_1|, \ldots , |x_n|), \qquad x\in (-1,1)^n.
\end{equation}

Using the result of Weisz \cite[Thm.~4.1]{Weisz} and the oddness of \( f \) in all its variables we obtain
\[
f = \sum_{\mathbf{m} \in \mathbb{N}^n} c_f(\mathbf{m}) e_{\mathbf{m}} \quad \text{in} \quad L_q(-1,1)^n,
\]
where
\[
c_f(\mathbf{k}) = \int_{(-1,1)^n} f(x) e_{\mathbf{k}}(x) \, dx, \quad \mathbf{k} = (k_1, \ldots, k_n) \in \mathbb{N}^n. 
\]
Consequently, we obtain \eqref{eq:1.2} with
\[
\hat{f}(\mathbf{k}) = 2^{-n} c_f(\mathbf{k}), \quad \mathbf{k} = (k_1, \ldots, k_n) \in \mathbb{N}^n.\qedhere
\]
\end{proof}

Now put $\varphi(t):=S(t)$, $t\in \mathbb{R}$, and according to \eqref{eq:0.1} consider for \(\mathbf{m} \in \mathbb{N}^n\) the functions 
\[
S_{\mathbf{m}}(x)=S(m_1x_1)\cdots S(m_nx_n), \quad x\in \mathbb{R}. 
\]
Following the ideas from  \cite{Lang17} we proceed similarly to the one-dimensional case: Again we construct a linear
homeomorphism $T$ of $L_{q}(0,1)^n$ onto itself, decomposing into a linear combination of certain isometries, that maps  $e_{\mathbf{m}}$ to
$S_{\mathbf{m}}$ for all $\mathbf{m}\in \mathbb{N}^n$.  From this we deduce  that
$(S_{\mathbf{m}})_{\mathbf{m}\in\N^n}$ forms a Schauder basis of $L_{q}(0,1)^n.$ 

To define the homeomorphism $T$, we need the following notation.
Given any function \( f \) on \([0,1)^n\), we extend it to \((-1,1)^n\) as in \eqref{extend-odd}. 
In the next step, $f$ is extended to $\R^n$ periodically, with period 2 in each variable, see Fig. \ref{fig-f-signs}. With a slight abuse of notation,
this extension is denoted by $f$ again.

\begin{figure}[h]
\begin{center}
\includegraphics[width=6cm]{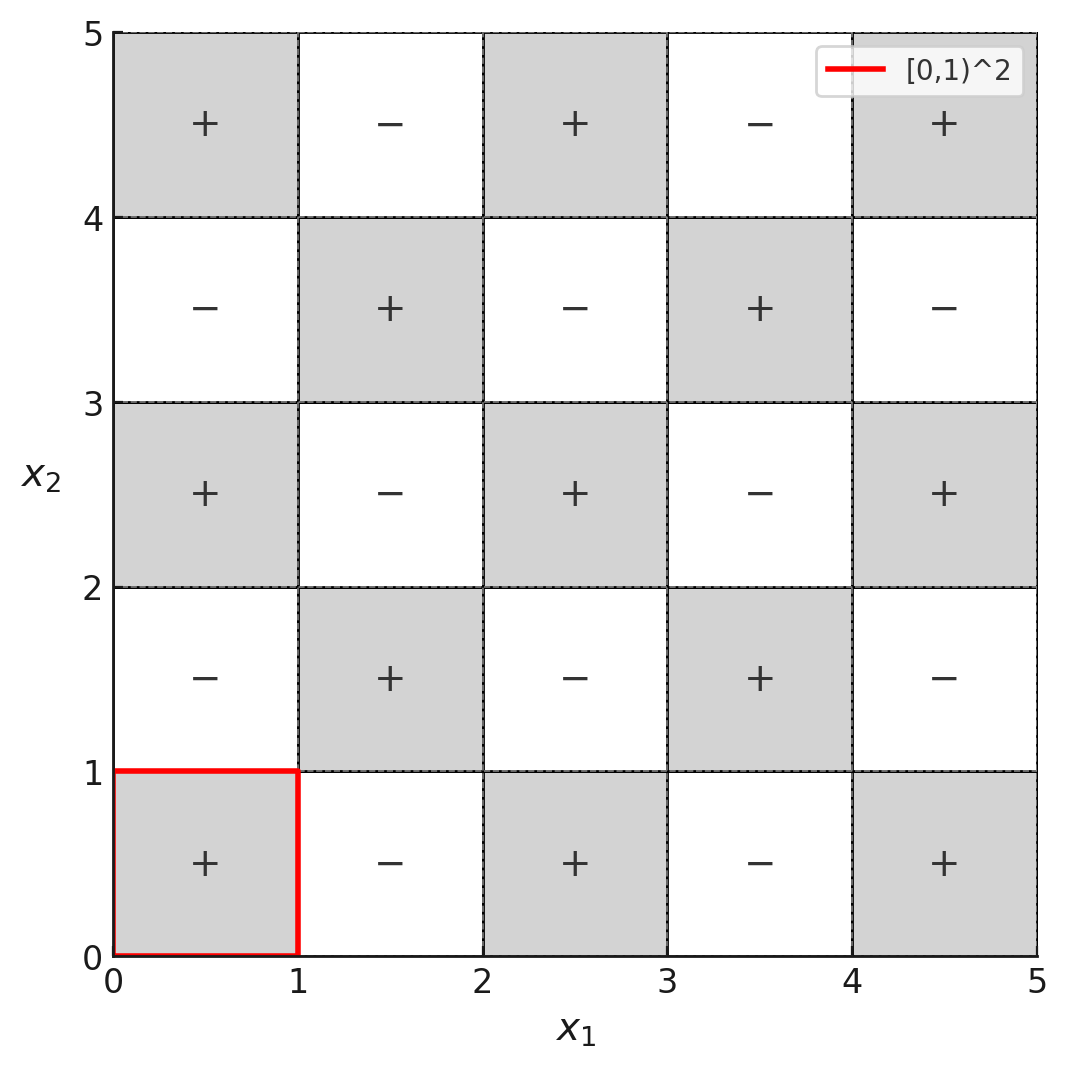}
\end{center}
\captionof{figure}{Illustration of the signs of $f$ in case $n=2$}
\label{fig-f-signs}
\end{figure}

Furthermore, the definition of $T$ features the Fourier coefficients of $S_{\mathbf{1}}(x)=S(x_1)\cdots S(x_n)$ as described in \eqref{eq:1.2}
and which are denoted and calculated as follows
\begin{equation}\label{eq:1.4}
\tau_{\mathbf{m}}:=\widehat{S_\mathbf{1}}(\mathbf{m})=\int_{(0,1)^n}S_1(x)e_{\mathbf{m}}(x)dx=\prod_{i=1}^n \int_0^1 S(x_i)e_{m_i}(x_i)dx_i=\prod_{i=1}^n \tau_{m_i}.
\end{equation}
Here, $\tau_{m_i}$ was defined in \eqref{Equation 8.0.1.3'} as the $m_i$-th Fourier coefficient of $S$.
Note that \(\widehat{S_{\mathbf 1}}(\mathbf{m}) = 0\) for every \(\mathbf{m} = (m_1, \ldots, m_n)\)
with some even \(m_i\), \(i \in \{1, \ldots, n\}\).

\begin{lemma}
\label{Lemma 2.2.6'} Let $q\in(1,\infty)$, $n\in \mathbb{N}$,  and define the map $T$   by%
\begin{equation}
Tg(x):=\sum_{\mathbf{m}\in \mathbb{N}^n}\tau_{\mathbf{m}}g(\mathbf{m}x), \qquad x\in (0,1)^n \label{Equation 8.0.1.10-2}%
\end{equation} 
(convergence is considered in the sense of \eqref{eq:0.2}). 
Then $T$ has the following properties: 
\bit 
\item[(i)] The map $T:L_{q}(0,1)^n\rightarrow$ $L_{q}(0,1)^n$ is bounded and linear  with $\left\Vert T\right\Vert \leq (\sqrt{2})^{-n}.$ 
\item[(ii)]  $Te_{\mathbf{m}}=S_{\mathbf{m}}$  for
all $\mathbf{m}\in\mathbb{N}^n$.  
\item[(iii)] $T$ has a bounded inverse if the coefficients $\tau_{\mathbf{m}}$ satisfy 
\begin{equation}\label{eq:1.7}
\sum_{\mathbf{k} \in \mathbb{N}^n, \, \mathbf{k} \neq \mathbf{1}} \left| \tau_{2\mathbf{k}-\mathbf{1}} \right| < |\tau_{\mathbf{1}}|.
\end{equation}
\eit 
\end{lemma}

\begin{proof}

\bit 
\item[(i)] Linearity is clear. Moreover, 
for  $g\in L_{q}(0,1)^n$  we see that 
\begin{align}\label{eq-isometry'}
\|g(\mathbf{m}\cdot)\|_{q}^q&=\int_{(0,1)^n}\left\vert g(\mathbf{m}x)\right\vert ^{q}dx  
 =\int_{(0,1)^n}\left\vert g(m_1x_1,\ldots, m_nx_n)\right\vert ^{q}dx  \notag\\
&  =\left(\prod_{i=1}^n m_i^{-1}\right)\int_{[0,m_1)\times \ldots [0,m_n)}%
\left\vert {g}(s)\right\vert ^{q}ds
=\int_{(0,1)^n}\left\vert g(s)\right\vert ^{q}ds=\|g\|_q^q. 
\end{align}
Thus, from \eqref{eq:1.4} and Proposition \ref{prop:1.5} we obtain
\begin{eqnarray*}
\left\Vert T\right\Vert &\leq &\sum_{\mathbf{k}\in\mathbb{N}^n}|\tau_{2\mathbf{k}-\mathbf{1}}|=
\left( \frac{4\sqrt{2}}{\pi^2} \right)^n
\sum_{\mathbf{k} \in \mathbb{N}^n} \frac{1}{(2\mathbf{k}-\mathbf{1})^2}\\
&=&  \left( \frac{4\sqrt{2}}{\pi^2} \right)^n \left(\sum_{k_1=1}^{\infty} \cdots \sum_{k_n=1}^{\infty} \frac{1}{(2k_1-1)^2 \cdots (2k_n-1)^2}\right)
\\
&=& \left( \frac{4\sqrt{2}}{\pi^2} \right)^n \left( \sum_{k=1}^{\infty} \frac{1}{(2k-1)^2} \right)^n
= \left( \frac{4\sqrt{2}}{\pi^2} \right)^n\left( \frac{\pi^2}{8} \right)^n=(\sqrt{2})^{-n}. 
\end{eqnarray*}
\item[(ii)] To show that $T e_{\mathbf{m}}=S_{\mathbf{m}}$ for every ${\mathbf{m}}\in\N^n$, we argue similarly as in \eqref{eq:T1}.
We employ Proposition \ref{prop:1.3} and calculate 
\begin{align*}
S_{\mathbf{m}}&=\sum_{\mathbf{k}\in\N^n}\widehat{S_{\mathbf{m}}}(\mathbf{k})e_{\mathbf{k}}
=\sum_{\mathbf{k}\in\N^n}e_{\mathbf{k}}\int_{(0,1)^n}S_{\mathbf{m}}(x)e_{\mathbf{k}}(x)dx
=\sum_{\mathbf{k}\in\N^n}e_{\mathbf{k}}\int_{(0,1)^n}S_{\mathbf{1}}(\mathbf{m}x)e_{\mathbf{k}}(x)dx\\
&=\sum_{\mathbf{k}\in\N^n}e_{\mathbf{k}}\int_{(0,1)^n} \sum_{\mathbf{l}\in\N^n} \widehat{S_{\mathbf{1}}}(\mathbf{l})e_{\mathbf{l}}(\mathbf{m}x)e_{\mathbf{k}}(x)dx
=\sum_{\mathbf{l}\in\N^n} \widehat{S_{\mathbf{1}}}(\mathbf{l}) \sum_{\mathbf{k}\in\N^n}e_{\mathbf{k}}\int_{(0,1)^n} e_{\mathbf{l}\mathbf{m}}(x)e_{\mathbf{k}}(x)dx\\
&=\sum_{\mathbf{l}\in\N^n} \widehat{S_{\mathbf{1}}}(\mathbf{l}) e_{\mathbf{l}\mathbf{m}}
=\sum_{\mathbf{l}\in\N^n} \tau_{\mathbf{l}} e_{\mathbf{m}}(\mathbf{l}\cdot)=Te_{\mathbf{m}}.
\end{align*}
\item[(iii)] 
By \eqref{eq:1.4}, $\tau_{\mathbf{1}}>0$.
To show the invertibility of $T$, we apply again the theory of Neumann series, cf. \cite[Theorem II.1.2]{Yos}.
For that sake, we rewrite $T$ as $T=\tau_{\mathbf{1}}(\text{Id}-M)$, where

\[
Mg(x) = - \frac{1}{\tau_{\mathbf{1}}}\sum_{\mathbf{k} \in \mathbb{N}^n, \, \mathbf{k} \neq \mathbf{1}} \tau_{2\mathbf{k}-\mathbf{1}}g((2\mathbf{k}-\mathbf{1})x), \qquad x\in (0,1)^n. 
\]
If condition \eqref{eq:1.7} holds, we obtain 
\[
\|M\| \leq \frac{1}{\tau_{\mathbf{1}}}\sum_{\mathbf{k} \in \mathbb{N}^n, \, \mathbf{k} \neq \mathbf{1}}
\left| \tau_{2\mathbf{k}-\mathbf{1}} \right| < 1
\]
and $T$ has a bounded inverse. \qedhere

\eit 
\end{proof}

\begin{theorem}\label{thm:123} Let $n \in  \{1,2,3\}$. Then the sequence $\left\{S_{\mathbf{k}}(\cdot)\right\}_{\mathbf{k}\in \mathbb{N}^n}$ is a Schauder basis in $L_q (0,1)^n$ for
any $q\in (1,\infty)$. 
\end{theorem}

\begin{proof} 
By Lemma \ref{Lemma 2.2.6'}(iii), the operator $T$ is a
homeomorphism and the assertion follows if we can show that 
\[
\sum_{\mathbf{k} \in \mathbb{N}^n, \, \mathbf{k} \neq \mathbf{1}} \left| \tau_{2\mathbf{k}-\mathbf{1}} \right| < |\tau_{\mathbf{1}}|.
\]
From  \eqref{eq:1.4}, Proposition \ref{prop:1.5}, and the observation 
\[
\sum_{\mathbf{k} \in \mathbb{N}^n, \, \mathbf{k} \neq \mathbf{1}} \frac{1}{(2\mathbf{k}-\mathbf{1})^2}
= -1 + \left( \sum_{k=1}^{\infty} \frac{1}{(2k-1)^2} \right)^n
= \left( \frac{\pi^2}{8} \right)^n - 1, 
\]
we obtain that 
\[
\sum_{\mathbf{k} \in \mathbb{N}^n, \, \mathbf{k} \neq \mathbf{1}} |\tau_{2\mathbf{k}-\mathbf{1}}|
= \left( \frac{4\sqrt{2}}{\pi^2} \right)^n \left( \left( \frac{\pi^2}{8} \right)^n - 1 \right).
\]
This quantity is smaller than
\[
\left( \frac{4\sqrt{2}}{\pi^2} \right)^n = \tau_{\mathbf{1}}
\]
for \( n = 1,2,3 \). 
\end{proof}

\subsection{Tensor products of Schauder bases}\label{sec:tensor}

It is a very natural question if the restriction to $n\in\{1,2,3\}$ in Theorem \ref{thm:123}
is necessary or if it is only an artifact of the proof method. Furthermore,
as the tensor product of two orthonormal bases of two Hilbert spaces is an orthonormal basis
of the tensor product of these Hilbert spaces (cf. \cite[Sect. II.4]{ReSi}), it is natural to hope that a similar statement
should be true for Schauder bases and tensor products of Banach spaces. 

The aim of this section is to shed light on these questions. We refer to \cite{cheney,Defant} or \cite{ryan}
for an extensive treatment of the theory of tensor products of Banach spaces.
If $X$ and $Y$ are two Banach spaces, then $X\otimes Y$ is their algebraic tensor product, i.e., the set of all finite sums
of the form $\displaystyle \sum_{j=1}^m x_j\otimes y_j$, where $x_j\in X$ and $y_j\in Y$ for all $j=1,\dots,m.$
The space $X\otimes Y$ may be equipped with different norms, which (after possible completion) can give rise
to different tensor product Banach spaces.

\medskip

If the sequence $(\varphi_j)_j$ is a Schauder basis of $X$ and $(\eta_l)_l$ is a Schauder basis of $Y$,
we would like to know if $(\varphi_j\otimes\eta_l)_{j,l}$ is a Schauder basis for the tensor product space of $X$ and $Y$.
According to \cite{GG,Lamadrid} or \cite[Prop. 4.25]{ryan} for general $X$ and $Y$ it turns out that $(\varphi_j\otimes\eta_l)_{j,l}$ (with the square ordering) is a Schauder basis
for $X{\hat\otimes}_{\pi}Y$ and $X{\hat\otimes}_{\varepsilon}Y$, where the tensor product spaces are equipped with the projective and injective norms, respectively.
In view of the Lebesgue spaces $L_q(0,1)$ this is not satisfactory, since the projective and injective tensor products of $L_q(0,1)$ with itself
do not coincide with the standard product space $L_q(0,1)^2$ for $q\in (1,\infty)$, $q\neq 2$. 

On the other hand, it is known (cf. \cite[Cor. 1.52]{cheney}) that 
\begin{equation}\label{p-nuclear-norm}
L_q(0,1)\ {\otimes}_{\alpha_q}L_q(0,1)=L_q(0,1)^2, 
\end{equation}
where the tensor product space is equipped with  the $q$-nuclear norm $\alpha_q$, see \cite[Def. 1.45]{cheney} or \cite[Section 6.2]{ryan}, where
this norm is denoted $g_p$ and called Chevet-Saphar tensor norm.

By \cite[Prop. 6.6]{ryan}, the $q$-nuclear norm is a  \emph{uniform crossnorm} (or reasonable crossnorm in notation of \cite{cheney}, see \cite[Lem. 1.46]{cheney}), which means that for
every bounded linear operator $S:X\to X$, every bounded linear operator $T:Y\to Y$
and every pair of $n$-tuples $\{f_1,\dots,f_n\}\subset X$ and $\{g_1,\dots,g_n\}\subset Y$ the following inequality holds
\[
\alpha_q\left(\sum_{i=1}^n (Sf_i)\otimes (Tg_i)\right)\le \|S\|\cdot \|T\|\cdot\alpha_q\left(\sum_{i=1}^n f_i\otimes g_i\right).
\]
This allows to define the tensor product of $S$ and $T$ on $X\otimes Y$ as
\[
(S\otimes T)\left(\sum_{i=1}^n f_i\otimes g_i\right):=\sum_{i=1}^n (Sf_i)\otimes (Tg_i)
\]
and to extend it by continuity to all $X\otimes_{\alpha_q}Y.$ Furthermore, $\|S\otimes T\|\le \|S\|\cdot \|T\|$.

As the next piece of the puzzle, it is known (cf. \cite{GG,Lamadrid} or \cite[Theorem 18.1]{Sing}), that if the tensor product space is generated by a uniform crossnorm,
then the tensor product of two Schauder bases is a Schauder basis of the tensor product space (if we keep the square-ordering).

Therefore, the following proposition is available in the literature on tensor product spaces if one combines several notions and sources.
For the sake of completeness, we summarize these findings and show directly that the tensor product
of two Schauder bases of $L_q(0,1)$ is a Schauder basis of $L_q(0,1)^2$.

\begin{proposition}
Let $1\le q<\infty$ and let $(\varphi_j)_{j\ge1}$ and $(\eta_l)_{l\ge1}$ be Schauder bases of $L_q(0,1)$.
Enumerate the elementary tensors $\varphi_j\otimes \eta_l$ by
the square ordering and write $(u_k)_{k\ge1}$ for the resulting sequence. Then $(u_k)_{k\ge 1}$ is a Schauder basis of $L_q(0,1)^2$.
\end{proposition}

\begin{proof}
Let $X=L_q(0,1)$. By the Schauder property, every $x\in X$ can be decomposed (in a unique way) as
\[
x=\sum_{j=1}^\infty a_j \varphi_j,
\]
which allows us to define the \emph{partial-sum projection} $P_N$ 
as
\[
P_Nx := \sum_{j=1}^N a_j\,\varphi_j. 
\]
Similarly, we define $Q_N$ associated to the Schauder basis $(\eta_l)_{l\ge 1}$.
The partial-sum projections are known to be uniformly bounded, i.e., they satisfy $\sup_N\|P_N\|=:C_1<\infty$ and $\sup_N\|Q_N\|=:C_2<\infty$.
This can be seen as follows: For each $x\in X$ the sequence $P_Nx$ converges to $x$ in norm,
hence $\sup_N\|P_Nx\|_q\leq \sup_N\|P_Nx-x\|_q+\|x\|_q<\infty$ for every $x\in X$. By the Uniform Boundedness Principle, this implies
$\sup_N\|P_N\|<\infty$. A similar argument holds for $Q_N$. 

Identify the algebraic tensor product $X\otimes X$ with the linear span of product functions in $L_q(0,1)^2$ via $g\otimes h\mapsto((s,t)\mapsto g(s)h(t))$.
Using the fact, see \cite[Lem. 10.15]{cheney}, that simple (finite) linear combinations of product functions are dense in 
$L_q(0,1)^2$, we obtain that
\begin{equation}\label{density-result}
\overline{L_q(0,1) \otimes L_q(0,1)}^{\|\cdot\|_{L_q(0,1)^2}} \;=\; L_q(0,1)^2, \qquad q \in [1,\infty), 
\end{equation}
where the closure is taken with respect to the norm in \(L_q(0,1)^2\).

For each $N\ge 1$ define the finite-rank operator $S_N=P_N\otimes Q_N$ on $X\otimes X$ by
\[
(P_N\otimes Q_N)\bigg(\sum_{k} g_k\otimes h_k\bigg)
:=\sum_k (P_N g_k)\otimes (Q_N h_k).
\]
These operators extend by continuity to bounded operators on $L_q(0,1)^2$ and satisfy 
\[
\|S_N\|=\|P_N\otimes Q_N\|\le\|P_N\|\|Q_N\|\le C_1C_2=:C.
\]
Observe, that $S_{N}(f)$ lies in ${\rm span}\{\varphi_j\otimes \eta_l:1\le j,l\le N\}$ for every $f\in X\otimes X$ and,
    by the continuity of the extension of $S_N$, also for every $f\in L_q(0,1)^2$. 
    Therefore, $S_{N+1}(f)$ ``extends'' $S_N(f)$ and its coefficients for $\varphi_j\otimes \eta_l$ with $1\le j,l\le N$ are the same. 
If $f=g\otimes h$ we can write
    \begin{align*}
        S_N(f)-f&=(P_N\otimes Q_N)(g\otimes h)-g\otimes h=(P_Ng)\otimes (Q_Nh)-g\otimes h\\
        &=-(g-P_Ng)\otimes (h-Q_Nh)-(P_Ng)\otimes(h-Q_Nh)-(g-P_Ng)\otimes Q_Nh
    \end{align*}
    and all three terms converge to zero (using that $P_Ng\to g$, $Q_Nh\to h$ and by the uniform boundedness of $\|P_N\|$ and $\|Q_N\|$).
    Therefore, $\lim_{N\to\infty}S_N(f)=f$. By linearity, the same is true for all $f\in X\otimes X$
    and by density, $\lim_{N\to\infty}S_N(F)=F$ follows for all $F\in L_q(0,1)^2$. 
    Indeed, for $F\in L_q(0,1)^2$, find $f\in X\otimes X$ with $\|F-f\|_q<\varepsilon.$ Then
    \[
    \limsup_{N\to\infty}\|S_N(F)-F\|_q\le \limsup_{N\to\infty}\|S_N(F)-S_N(f)\|_q+\limsup_{N\to\infty}\|S_N(f)-f\|_q+\limsup_{N\to\infty}\|f-F\|_q\le (C+1)\varepsilon.
    \]
Moreover, uniqueness follows in the following way:  Assume that
    \begin{equation}\label{eq:unique}
    \lim_{N\to\infty}\sum_{j,l\le N} a_{j,l} \varphi_j\otimes \eta_l=0.
    \end{equation}
    Then we get for every $M,N\ge 1$
    \begin{align*}
    \left\|\sum_{j,l\le N}a_{j,l} \varphi_j\otimes \eta_l\right\|_q&=
    \left\|\sum_{j,l\le N}a_{j,l} (P_N\varphi_j)\otimes (Q_N\eta_l)\right\|_q=
    \left\|S_N\left(\sum_{j,l\le M+N}a_{j,l} \varphi_j\otimes \eta_l\right)\right\|_q\\
    &\le C\left\|\sum_{j,l\le N+M}a_{j,l} \varphi_j\otimes \eta_l\right\|_q.
    \end{align*}
    Taking the limit $M\to\infty$, we obtain that 
    \[
    \left\|\sum_{j,l\le N}a_{j,l} \varphi_j\otimes \eta_l\right\|_q=0,
    \]
    i.e., that $a_{j,l}=0$ for all $1\le j,l\le N$. As this holds for every $N\ge 1$, we observe that \eqref{eq:unique}
    implies that $a_{j,l}=0$ for all $j,l\ge 1.$
\end{proof}
\begin{corollary}
The sequence $\left\{S_{\mathbf{k}}(\cdot)\right\}_{\mathbf{k}\in \mathbb{N}^n}$ (with the square-ordering) is a Schauder basis in $L_q (0,1)^n$ for
any $q\in (1,\infty)$ and any $n\ge 1$. 
\end{corollary}

\subsection{Bases with the inner product structure}

The concept of tensor products is surely important and in some sense standard if we want to transfer results from univariate functions to functions of many variables.
It was, however, one of the main achievements of \cite{SV} to show that replacing tensor products by inner products can bring essential advantages.
Recall, that it was shown there, that the system ${\mathcal R}_n$, defined in \eqref{eq:CS_Rn_intro} as
\begin{equation}\label{eq:CS_Rn}
 {\mathcal R}_n:=\{1\}\cup\{\sqrt{3}\, \CalC_{\mathbf{k}},\sqrt{3}\,\CalS_{\mathbf{k}}: \mathbf{k}\in\Z^n, \mathbf{k}+\!\!\!>0\}
\end{equation}
is a Riesz basis of $L_2(0,1)^n$ with constants $A=1/2$ and $B=3/2$. The aim of this section is to study the behavior of \eqref{eq:CS_Rn}
outside of the scale of Hilbert spaces. 

\medskip

We use the same approach as we used in the proof of Theorem \ref{Theorem 8.0.7.6}.
First, we consider the trigonometric analogue of \eqref{eq:CS_Rn} and show in Theorem \ref{thm:sincos_inner}, that it constitutes an orthonormal basis in $L_2(0,1)^n.$
Then we derive, that it constitutes also a Schauder basis of every $L_q(0,1)^n$, $1<q<\infty.$
Afterwards we construct a bounded and invertible linear mapping $T$ of $L_q(0,1)^n$ onto itself, which maps this trigonometric basis onto \eqref{eq:CS_Rn},
similarly to Lemma \ref{Lemma 2.2.6} and Lemma \ref{Lemma 2.2.6'}.

\begin{theorem}\label{thm:sincos_inner}
The system
\begin{equation}\label{eq:sincos_inner}
\{1\}\cup\{\sqrt{2}\,\cos(2\pi \mathbf{k}\cdot x):\mathbf{k}\in\Z^n, \mathbf{k}+\!\!\!>0\}\cup\{\sqrt{2}\,\sin(2\pi \mathbf{k}\cdot x):\mathbf{k}\in\Z^n,\mathbf{k}+\!\!\!>0\}
\end{equation}
is an orthonormal basis of the real Hilbert space $L_2(0,1)^n$.
\end{theorem}
\begin{proof}
We use that $\{e^{2\pi i\mathbf{k}\cdot x}:k\in\Z^n\}$ is an orthonormal basis of the complex Hilbert space $L_2(0,1)^n$.
First, let us consider any non-zero $\mathbf{k}\in\Z^n\setminus\{0\}$. Taking the real and imaginary part of $\langle e^{2\pi i\mathbf{k}\cdot x}, 1\rangle=0$,
we obtain that the constant function is orthogonal to the remaining elements of \eqref{eq:sincos_inner}.
Applying the same idea to the equation $\langle e^{2\pi i\mathbf{k}\cdot x}, e^{2\pi i\mathbf{k}\cdot x}\rangle=1$,
we get that $\cos(2\pi \mathbf{k}\cdot x)$ and $\sin(2\pi \mathbf{k}\cdot x)$ are orthogonal for every $\mathbf{k}\in\Z^n\setminus\{0\}$
and, by symmetry, that $\sqrt{2}$ is the correct normalization factor in \eqref{eq:sincos_inner}.
Similarly, for non-zero
$\mathbf{k},\mathbf{l}\in\Z^n\setminus\{0\}$ with $\mathbf{k}\not=\mathbf{l}$,
we separate the real and imaginary part of $\langle e^{2\pi i\mathbf{k}\cdot x}, e^{2\pi i\mathbf{l}\cdot x}\rangle=\langle e^{2\pi i\mathbf{k}\cdot x}, e^{-2\pi i\mathbf{l}\cdot x}\rangle=0$.
We obtain a system of four (real) equations, which shows that all the terms in \eqref{eq:sincos_inner} are orthogonal to each other.

\medskip

To show that \eqref{eq:sincos_inner} is also a basis, we take a real function $f\in L_2(0,1)^n$ and consider its unique decomposition
\begin{equation}\label{eq:f_Four}
f=\sum_{\mathbf{k}\in\Z^n}\alpha_{\mathbf{k}}e^{2\pi i\mathbf{k}\cdot x}.
\end{equation}
Then the conjugate of $f$ has the Fourier series
\begin{equation}\label{eq:f_Four'}
\bar f=\sum_{\mathbf{k}\in\Z^n}\bar{\alpha}_{\mathbf{k}}e^{-2\pi i\mathbf{k}\cdot x}.
\end{equation}
As $f=\bar f$, we get $\alpha_{-\mathbf{k}}=\bar\alpha_{\mathbf{k}}$ by comparing \eqref{eq:f_Four} and \eqref{eq:f_Four'}. That allows us to write
\begin{align*}
  f&=\alpha_0+\sum_{\mathbf{k}+\!\!>0}\alpha_{\mathbf{k}}e^{2\pi i\mathbf{k}\cdot x}+\sum_{\mathbf{k}+\!\!>0}\alpha_{\mathbf{-k}}e^{-2\pi i\mathbf{k}\cdot x}\\
  &=\alpha_0+\sum_{\mathbf{k}+\!\!>0}\cos(2\pi\mathbf{k}\cdot x)\cdot [\alpha_{\mathbf{k}}+\bar\alpha_{\mathbf{k}}]
  +\sin(2\pi\mathbf{k}\cdot x)[i\alpha_{\mathbf{k}}+\overline{i\alpha_{\mathbf{k}}}],
\end{align*}
which is a decomposition of $f$ into \eqref{eq:sincos_inner} with real coefficients.
\end{proof}

\begin{proposition}\label{prop:sincos}
Let $n\ge 1$ and $1<q<\infty$. Then the system \eqref{eq:sincos_inner} is a Schauder basis of the Lebesgue space $L_q(0,1)^n$.
\end{proposition} 
\begin{proof}
This follows directly from the properties of multivariate trigonometric series \cite[Theorem 4.1]{Weisz}.
Indeed, this result shows that if $\mathbf{k}\in\N^n$ and 
\begin{equation}\label{eq:partialF}
(S^{\mathbf{k}} f)(x)=\sum_{|\ell_1|\le k_1}\dots \sum_{|\ell_n|\le k_n}\hat f(\boldsymbol{\ell}) e^{i\boldsymbol{\ell}\cdot x}
\end{equation}
denotes the rectangular partial sum of the Fourier series of $f\in L_1(\T^n)$,
then
\[
\lim_{\mathbf{k}\to\infty} S^{\mathbf{k}}f=f
\]
in the $L_q$-norm for every $1<q<\infty$. The convergence $\mathbf{k}\to\infty$ means that $\min(k_1,\dots,k_n)\to\infty$,
i.e., the convergence in the Pringsheim sense. By splitting the complex exponential, we may rewrite \eqref{eq:partialF} as
\begin{align*}
(S^{\mathbf{k}} f)(x)&=\sum_{|\ell_1|\le k_1}\dots \sum_{|\ell_n|\le k_n}\hat f(\boldsymbol{\ell}) e^{i\boldsymbol{\ell}\cdot x}\\
&=\hat f(0)+\mkern-30mu\sum_{\substack{\boldsymbol{\ell} +\!\!>0\\|\ell_1|\le k_1,\dots,|\ell_n|\le k_n}} \mkern-30mu[\alpha_{\boldsymbol{\ell}}\cos(\boldsymbol{\ell}\cdot x)+\beta_{\boldsymbol{\ell}} \sin(\boldsymbol{\ell}\cdot x)],
\end{align*}
where $\alpha_{\boldsymbol{\ell}}=\hat f(\boldsymbol{\ell})+\hat f(-\boldsymbol{\ell})$ and $\beta_{\boldsymbol{\ell}}=i[\hat f({\boldsymbol{\ell}})-\hat f(-\boldsymbol{\ell})].$
\end{proof}

\begin{theorem}\label{thm:Rn}
Let $n\ge 1$ and $q\in(1,\infty).$ Then the system ${\mathcal R}_n$ is a Schauder basis of $L_q(0,1)^n$.
\end{theorem}
\begin{proof}
We apply again the argument of \cite[p. 75]{Hig}, sketched and applied already in Section \ref{sec:bases_leb}.
Let us denote $c_j(t):=\cos(2\pi jt)$ and $s_j(t):=\sin(2\pi jt)$ for $j\ge 1.$
Furthermore, for $\mathbf{k}\in\Z^n$, we put $c_{\mathbf{k}}(x):=\cos(2\pi \mathbf{k}\cdot x)$ and $s_{\mathbf{k}}(x):=\sin(2\pi \mathbf{k}\cdot x)$.
By Proposition \ref{prop:sincos}, the system $\{1\}\cup\{c_{\mathbf{k}},s_{\mathbf{k}}:\mathbf{k}\in\Z^n, \mathbf{k}+\!\!\!>0 \}$
is the Schauder basis of $L_q(0,1)^n$.

We use the Fourier decomposition of $\CalC$ and $\CalS$
\begin{equation}\label{eq:CS_Fourier}
\CalC=\frac{8}{\pi^2}\sum_{m=0}^\infty \frac{1}{(2m+1)^2} c_{2m+1}\quad \text{and}\quad
\CalS=\frac{8}{\pi^2}\sum_{m=0}^\infty \frac{(-1)^m}{(2m+1)^2} s_{2m+1}.
\end{equation}
For $f\in L_q(0,1)^n$ defined on $[0,1)^n$, we denote by $f$ also its 1-periodic extension onto $\R^n.$
We now define a bounded linear and invertible operator $T:L_q(0,1)^n\to L_q(0,1)^n$ in such a way that it maps \eqref{eq:sincos_inner} onto \eqref{eq:CS_Rn}.
For that sake, we put 
\begin{equation}\label{eq:defT}
T(f)(x)=\frac{8}{\pi^2}\sum_{j=0}^\infty \frac{f((4j+1)x)}{(4j+1)^2}+\frac{8}{\pi^2}\sum_{j=0}^\infty \frac{f((4j+3)(\mathbf{1}-x))}{(4j+3)^2}.
\end{equation}
With this definition, we obtain
\[
T(1)=\frac{8}{\pi^2} \sum_{j=0}^\infty \frac{1}{(2j+1)^2}=1
\]
and, using \eqref{eq:CS_Fourier}, also
\begin{align*}
T(c_{\mathbf{k}})(x)&=\frac{8}{\pi^2}\sum_{j=0}^\infty \frac{c_{\mathbf{k}}((4j+1)x)}{(4j+1)^2}+\frac{8}{\pi^2}\sum_{j=0}^\infty \frac{c_{\mathbf{k}}((4j+3)(\mathbf{1}-x))}{(4j+3)^2}\\
&=\frac{8}{\pi^2}\sum_{j=0}^\infty \frac{c_{4j+1}(\mathbf{k}\cdot x)}{(4j+1)^2}+\frac{8}{\pi^2}\sum_{j=0}^\infty \frac{c_{4j+3}(\mathbf{k}\cdot(\mathbf{1}-x))}{(4j+3)^2}\\
&=\frac{8}{\pi^2}\sum_{j=0}^\infty \frac{c_{4j+1}(\mathbf{k}\cdot x)}{(4j+1)^2}+\frac{8}{\pi^2}\sum_{j=0}^\infty \frac{c_{4j+3}(\mathbf{k}\cdot x)}{(4j+3)^2}=\CalC(\mathbf{k}\cdot x)=\CalC_{\mathbf{k}}(x).
\end{align*}
The same calculation with $s_{\mathbf{k}}$ instead of $c_{\mathbf{k}}$ and with $s_{4j+3}(\mathbf{k}\cdot(\mathbf{1}-x))=-s_{4j+3}(\mathbf{k}\cdot x)$ gives
$T(s_{\mathbf{k}})(x)=\CalS(\mathbf{k}\cdot x)=\CalS_{\mathbf{k}}(x).$

The boundedness and invertibility of $T$ follows similarly as in Lemma \ref{Lemma 2.2.6}, using the technique of Neumann series.
Indeed,  we observe that the first summand in \eqref{eq:defT} is equal to $\displaystyle \frac{8}{\pi^2}f(x)=\frac{8}{\pi^2}\,\text{Id}(f)(x)$.
We therefore rewrite \eqref{eq:defT} as $\displaystyle \frac{8}{\pi^2}T=\text{Id}-M$, where
\begin{equation}\label{eq:Tnorm}
\|M\|=\left\|\frac{\pi^2}{8}T-\text{Id}\right\|\le \sum_{j=1}^\infty\frac{1}{(2j+1)^2}=\frac{\pi^2}{8}-1<1.
\end{equation}
We conclude, that $\text{Id}-M$ is invertible and, therefore, also $T$ is invertible. 
Finally, the result follows using the fact that \eqref{eq:sincos_inner} is a Schauder basis of $L_q(0,1)^n.$
\end{proof}
Again, applying the same method to $q=2$ leads to an improvement of the lower Riesz constant of ${\mathcal R}_d$.
\begin{theorem}\label{thm:Riesz_n}
 Let $n\ge 1$. Then ${\mathcal R}_n$ is a Riesz basis of $L_2(0,1)^n$ with Riesz constants $A=0.5787...$ and $B=3/2.$
\end{theorem}
\begin{proof}
The proof is based on the same idea as the proof of Theorem \ref{thm:Riesz_1}. As the constant function is in both \eqref{eq:sincos_inner} and ${\mathcal R}_n$
and is orthogonal to the rest of these two systems, we may leave out the constant function and only work on its complement in $L_2(0,1)^n.$
On this complement, we define $R:=\sqrt{3}/\sqrt{2}\, \cdot T.$ Then $B:=\|R\|^2\, \le \, 3/2$ since 
\[
\|T\|\le \frac{8}{\pi^2}\,\sum_{j=0}^\infty\frac{1}{(2j+1)^2}=1.
\]

To estimate the lower Riesz constant $A$, we rewrite again $R$ as $R=a\, \mathrm{Id}-M$, where
\[
a=\frac{\sqrt{3}}{\sqrt{2}}\cdot\frac{8}{\pi^2}\quad\text{and}\quad
-Mf(x)=\frac{\sqrt{3}}{\sqrt{2}}\cdot \frac{8}{\pi^2}\Biggl(\sum_{j=1}^\infty \frac{f((4j+1)x)}{(4j+1)^2}+\sum_{j=0}^\infty \frac{f((4j+3)(\mathbf{1}-x))}{(4j+3)^2}\Biggr).
\]
Then
\begin{align*}
    \|M\|\le \frac{\sqrt{3}}{\sqrt{2}}\cdot \frac{8}{\pi^2}\Bigl(\frac{\pi^2}{8}-1\Bigr)<\frac{\sqrt{3}}{\sqrt{2}}\cdot \frac{8}{\pi^2}=a
\end{align*}
and we obtain (exactly as in the proof of Theorem \ref{thm:Riesz_1})
\[
A\ge (a-\|M\|)^2\geq \left\{\frac{4\sqrt{6}}{\pi^2}\left(2-\frac{\pi^2}{8}\right)\right\}^2=0.5787\dots.
\]
\end{proof}

\end{document}